\documentclass[12pt]{amsart}
\newtheorem{thm}{Theorem}[section]

\newtheorem{prop}[thm]{Proposition}

\theoremstyle{remark}

\newtheorem{exa}[thm]{Example}

\newtheorem*{acknowledgement}{Acknowledgment}
\makeatletter
 
 \@addtoreset{equation}{section}

\makeatother

\title{Strange examples of local signatures for fibered surfaces of small genus}
\author{Makoto Enokizono}
\address{Makoto Enokizono,
	Department of Mathematics,
	Graduate School of Science,
	Osaka University,
	Toyonaka, Osaka 560-0043, Japan}
\email{m-enokizono@cr.math.sci.osaka-u.ac.jp}
\usepackage[top=35truemm,bottom=35truemm,left=25truemm,right=25truemm]{geometry}
\usepackage{amsmath,cases}
\usepackage{amsfonts}
\usepackage[all]{xy}
\usepackage{here}
\usepackage{latexsym}
  
\begin{document}
\maketitle

\begin{abstract}
We give examples of local signatures, completely different from the usual ones, for general fibrations of genus $2$ and genus $3$.
\end{abstract}

\section*{Introduction}
For a closed oriented real $4$-manifold $X$, the {\em signature of $X$} is defined to be the signature of the intersection form $H^2(X,\mathbb{R})\times H^2(X,\mathbb{R})\to \mathbb{R}$, which is a symmetric bilinear form.
We consider the situation that $X$ admits a fibration $f\colon X\to B$ over a closed oriented real surface $B$.
Under some conditions, the signature of $X$ happens to localize around a finite number of fiber germs $F_1$, $F_2$, \ldots, $F_m$:
$$
\mathrm{Sign}(X)=\sum_{i=1}^{m}\sigma(F_i).
$$
We call this phenomenon a {\em localization of the signature} and the value $\sigma(F_i)$ a {\em local signature of $F_i$}.
A first example of local signatures is the one for genus $1$ fibrations due to Matsumoto~\cite{Ma}.
He also gave a local signature for Lefschetz fibrations of genus $2$ in \cite{Ma2},
which was generalized by Endo~\cite{En} for hyperelliptic fibrations.
Later, Kuno~\cite{Ku} defined a local signature for non-hyperelliptic fibrations of genus $3$.
On the other hand, Horikawa~\cite{Ho} defined a function $\mathrm{Ind}(F)$ on the set of holomorphic fiber germs $F$ of genus $2$, which is nowadays called a {\em Horikawa index}, in order to study algebraic surfaces of general type near the Noether line.
Once a Horikawa index is defined (for a certain type of holomorphic fibrations), we can define a local signature by using it, as shown in \cite{ak}.
After Horikawa's work, Xiao~\cite{pi1} and Arakawa-Ashikaga~\cite{ArAs} defined a Horikawa index and a local signature for hyperelliptic fibrations.
Terasoma~\cite{Te} showed the coincidence of Endo's local signature and Arakawa-Ashikaga's one.
For non-hyperelliptic genus $3$ fibrations, Reid~\cite{Re} defined a Horikawa index.
Similarly to Terasoma's proof, Kuno's local signature and Reid's one for non-hyperelliptic fibrations of genus $3$ also coincide (cf.\ \cite{As2}).

In this short note, in the algebro-geometric category, we construct a local signature associated with an effective divisor $D$ on the moduli space $\mathcal{M}_g$ of smooth curves of genus $g$
and compute some examples of local signatures for general fibrations of genus $2$ or $3$, which are different from Endo-Arakawa-Asikaga's one and Kuno-Reid's one.
The idea of constructions is essentially due to Ashikaga-Yoshikawa \cite{AsYo},
who called the divisor $4\lambda-\delta$ on the moduli space $\overline{\mathcal{M}}_g$ of stable curves of genus $g$ the {\em signature divisor} and gave a local signature by pulling back the signature divisor using a geometric meaningful effective divisor $D$, e.g., the Brill-Noether locus, via the moduli map of a fiber germ.
Replacing $D$ by another effective divisor, the associated local signature varies.
We compute local signatures in the case that $g=2$ and $D$ is the bielliptic locus and that $g=3$ and $D$ is the locus of curves having a hyperflex.

\begin{acknowledgement}
I would like to express special thanks to Prof.\ Kazuhiro Konno for many comments and supports.
Thanks are also due to Prof.\ Tadashi Ashikaga for useful advises and discussions.
The research is supported by JSPS KAKENHI No.\ 16J00889.
\end{acknowledgement}

\section{Local signature associated with an effective divisor on $\mathcal{M}_g$}

Let $\mathcal{M}_g$ and $\overline{\mathcal{M}}_g$ respectively denote the moduli space of smooth curves of genus $g$ and the moduli space of stable curves of genus $g$.
The rational Picard group of $\overline{\mathcal{M}}_g$ is generated freely by the Hodge bundle $\lambda$ and the boundary divisors $\delta_0,\delta_1,\ldots,\delta_{[g/2]}$ for $g\ge 3$, where we use the notation in \cite{HaMo}.
 Let $D$ be an effective divisor on $\mathcal{M}_g$ and $\overline{D}$ the compactification of $D$ in $\overline{\mathcal{M}}_g$.
Then we can write $\overline{D}\sim_{\mathbb{Q}}a\lambda-\sum_{i=0}^{[g/2]}b_i\delta_i$ for some rational numbers $a, b_i>0$, where the symbol $\sim_{\mathbb{Q}}$ means the $\mathbb{Q}$-linear equivalence.

Let $f\colon S\to B$ be a surjective morphism from a complex smooth projective surface $S$ to a smooth projective curve $B$ whose general fiber $F_{\eta}$ is a smooth projective curve of genus $g$, which is called a fibered surface or a global fibration of genus $g$.
Let $K_f=K_S-f^{*}K_B$ denote the relative canonical bundle of $f$ and put
$$
\chi_f=\mathrm{deg}f_{*}\mathcal{O}_S(K_f)=\chi(\mathcal{O}_S)-(g-1)(b-1),
$$
$$
e_f=e_{\mathrm{top}}(S)-e_{\mathrm{top}}(F_{\eta})e_{\mathrm{top}}(B)=e_{\mathrm{top}}(S)-4(g-1)(b-1),
$$
where $b$ is the genus of $B$ and $e_{\mathrm{top}}(X)$ the topological Euler number of $X$.

Let $f\colon S\to \Delta$ be a relatively minimal degeneration of curves of genus $g$, that is,  
$f$ is a surjective proper morphism from a complex smooth surface $S$ to a small open disk $\Delta$ such that $f^{-1}(t)$ is a smooth curve of genus $g$ for any $t\neq 0$ and the central fiber $F:=f^{-1}(0)$ has no $(-1)$-curves.
We take the stable reduction $\widetilde{f}\colon \widetilde{S}\to \widetilde{\Delta}$ of $f$ via $\widetilde{\Delta}\to \Delta;\ z\mapsto z^{N}$.
Resolving singularities of $\widetilde{S}$, we obtain a semi-stable reduction $\widehat{f}\colon \widehat{S}\to \widetilde{\Delta}$.
Note that $N$ can be taken as the pseudo-period of the topological monodromy $\mu_f$ of $f$ as a pseudo-periodic class (cf.\ \cite{As}).
Put $F:=f^{-1}(0)$ and $\widehat{F}:=\widehat{f}^{-1}(0)$.
Let 
$$
\mathrm{Lsd}(F):=\sigma(f,F;h_{\partial S})-\frac{1}{N}\sigma(\widehat{f},\widehat{F};h_{\partial \widehat{S}})
$$
be the local signature defect of $(f,F)$ (more precisely, see \cite{As})
and
$$
e_{F}:=\left(e_{\mathrm{top}}(F)-(2-2g)\right)-\frac{1}{N}\left(e_{\mathrm{top}}(\widehat{F})-(2-2g)\right).
$$
On the other hand, the local invariants $c_1^{2}(F)$, $c_2(F)$ and $\chi_{F}$ were defined in \cite{Tan2} for a fiber germ $F$ of a global fibration $f\colon S\to B$. Indeed,
\begin{prop}
We have $e_{F}=c_2(F)$ and
$$
\mathrm{Lsd}(F)=\frac{1}{3}(c_1^{2}(F)-2e_{F})=4\chi_{F}-e_{F}.
$$
\end{prop}

\begin{proof}
These invariants satisfy the following properties:
Let $f\colon S\to B$ be a fibered surface of genus $g$ and $\widehat{f}\colon \widehat{S} \to \widetilde{B}$ be the semi-stable reduction of $f$ via a cyclic covering $\widetilde{B}\to B$ of degree $N$.
Then we have
\begin{align}
\mathrm{Sign}(S)-\frac{1}{N}\mathrm{Sign}(\widehat{S})&=\sum_{p\in B}\mathrm{Lsd}(F_p), \nonumber \\
K_f^2-\frac{1}{N}K_{\widehat{f}}^2&=\sum_{p\in B}c_1^2(F_p), \nonumber \\
e_f-\frac{1}{N}e_{\widehat{f}}&=\sum_{p\in B}c_2(F_p)=\sum_{p\in B}e_{F_p}, \label{redeq} \\
\chi_f-\frac{1}{N}\chi_{\widehat{f}}&=\sum_{p\in B}\chi_{F_p}. \nonumber
\end{align}
Let $F$ be an arbitrary fiber germ in a global fibration $f\colon S\to B$.
Taking base change, we may assume that any fiber of $f$ other than $F$ is semi-stable.
Thus we get the assertion from Hirzebruch's signature formula $\mathrm{Sign}(S)=K_f^2-8\chi_f$, Noether's formula $12\chi_f=K_f^2+e_f$ and \eqref{redeq} since $\mathrm{Lsd}(\widehat{F})=c_1^2(\widehat{F})=c_2(\widehat{F})=e_{\widehat{F}}=\chi_{\widehat{F}}=0$
for any semi-stable fiber germ $\widehat{F}$.
\end{proof}

Let $\rho_{\widehat{f}}\colon \widetilde{\Delta}\to \overline{\mathcal{M}}_g$ be the moduli map of the 
semi-stable reduction $\widehat{f}\colon \widehat{S}\to \widetilde{\Delta}$.
For an effective divisor $E$ on $\overline{\mathcal{M}}_g$ not containing the image $\rho_{\widehat{f}}(\widetilde{\Delta})$, we can define the pull-back $\rho_{\widehat{f}}^{*}E$. Let $E(\widehat{F}):=\mathrm{deg}(\rho_{\widehat{f}}^{*}E$).
Note that even when $E\sim E'$ holds for two effective divisors $E$ and $E'$, it is not always true that $E(\widehat{F})=E'(\widehat{F})$ because we treat local fibrations here.
Given an effective divisor $D$ on $\mathcal{M}_g$ such that $\overline{D}$ does not contain $\rho_{\widehat{f}}(\widetilde{\Delta})$ with $\overline{D}\sim_{\mathbb{Q}}a\lambda-\sum_{i=0}^{[g/2]}b_i\delta_i$, we put
$$
\lambda_{D}(\widehat{F}):=\frac{1}{a}\left(\overline{D}(\widehat{F})+\sum_{i}b_i\delta_i(\widehat{F})\right).
$$
In general, for a relatively minimal fiber germ $F$, we define
$$
\lambda_{D}(F):=\chi_{F}+\frac{\lambda_{D}(\widehat{F})}{N}
$$

and

$$
\delta(F):=e_{F}+\frac{\delta(\widehat{F})}{N}=e_{\mathrm{top}}(F)-(2-2g),
$$
which are independent of the choice of $N$.

Now we consider a global fibration $f\colon S\to B$, that is, a surjective morphism from a smooth projective surface $S$ to a smooth projective curve $B$ with connected fibers.
Assume that the moduli point of the general fiber of $f$ is not contained in $D$.
From \eqref{redeq}, we have
$$
\chi_f=\sum_{p\in B}\lambda_D(F_p),\quad e_f=\sum_{p\in B}\delta(F_p).
$$

From Hirzebruch's signature formula $\mathrm{Sign}(S)=4\chi_f-e_f$, we can write
$$
\mathrm{Sign}(S)=\sum_{p\in B}(4\lambda_{D}(F_p)-\delta(F_p)).
$$
We call $\sigma_{D}(F):=4\lambda_{D}(F)-\delta(F)$ the {\em local signature of a fiber germ $F$ associated with $D$}.
Note that the divisor $4\lambda-\delta$ is called the signature divisor in \cite{AsYo}.

\section{Examples}

Now we consider two effective divisors $E_{g,-1}$ and $E_{g,1}$ on $\mathcal{M}_g$, which parameterize curves $C$ of genus $g$ having a special Weierstrass point.
Let $C$ be a smooth curve of genus $g$.
Let $p$ be a Weierstrass point of $C$, i.e., a point on $C$ satisfying $h^{0}(gp)\ge 2$.
Then $p$ is said to be {\em exceptional of type $g-1$} (resp.\ {\em of type $g+1$}) if $h^{0}((g-1)p)\ge 2$ (resp.\ $h^{0}((g+1)p)\ge 3$).
The locus $E_{g,-1}$ (resp.\ $E_{g,1}$) on $\mathcal{M}_g$ is (roughly) defined by the set of curves of genus $g$ with an exceptional Weierstrass point of type $g-1$ (resp.\ of type $g+1$) with the natural scheme structure, which is of codimension $1$ for $g\ge 3$. For more details, see \cite{Di}.
For $g=2$, the loci $\overline{E}_{2,-1}$ and $\overline{E}_{2,1}$ are empty.
For $g=3$, $\overline{E}_{3,-1}$ is coincide with the hyperelliptic locus $\overline{\mathcal{H}}_3$ as a set, but as a divisor, we have $\overline{E}_{3,-1}=8\overline{\mathcal{H}}_3$.
Indeed, once a genus $3$ curve has one exceptional Weierstrass point of type $2$, it becomes hyperelliptic and hence has $8$ Weierstrass points of type $2$ automatically.
Since the hyperelliptic Weierstrass point is exceptional of type $g-1$ and $g+1$, 
the hyperelliptic locus $\overline{\mathcal{H}}_g$ is contained in both $\overline{E}_{g,-1}$ and $\overline{E}_{g,1}$.
In particular, $\overline{E}_{3,-1}=8\overline{\mathcal{H}}_3$ is a subdivisor of $\overline{E}_{3,1}$.
Thus we can define an effective divisor $\overline{\mathcal{HF}}:=\overline{E}_{3,1}-\overline{E}_{3,-1}$.
As a different definition, let $\mathcal{HF}$ be the locus on the moduli space $\mathcal{M}_3\setminus \mathcal{H}_3$ of smooth plane quartics parameterizing plane quartic curves with a hyperflex, i.e., $4$-fold tangent point.
Then the above $\overline{\mathcal{HF}}$ is just the closure of $\mathcal{HF}$ in $\overline{\mathcal{M}}_3$.
The locus $\overline{\mathcal{HF}}$ has multiplicity $1$ around general points.
For $g\ge 4$, $\overline{E}_{g,-1}$ and $\overline{E}_{g,1}$ also have multiplicity $1$ around general points.
It is known that the rational divisor classes of $\overline{E}_{g,-1}$ and $\overline{E}_{g,1}$ are given by
\begin{align*}
\overline{E}_{g,-1}&=\frac{g^2(g-1)(3g-1)}{2}\lambda-\frac{(g-1)^2g(g+1)}{6}\delta_0-\sum_{i=1}^{[g/2]}\frac{i(g-i)g(g^2+g-4)}{2}\delta_i, \\
\overline{E}_{g,1}&=\frac{(g+1)(g+2)(3g^2+3g+2)}{2}\lambda-\frac{g(g+1)^2(g+2)}{6}\delta_0-\sum_{i=1}^{[g/2]}\frac{i(g-i)(g+1)(g+2)^2}{2}\delta_i
\end{align*}
(cf.\ \cite{Di}, \cite{Cu}, \cite{CumEsGa}).
In particular, we have
\begin{align*}
\overline{E}_{3,-1}&=72\lambda-8\delta_0-24\delta_1, \quad \overline{E}_{3,1}=380\lambda-40\delta_0-100\delta_1, \\
\overline{\mathcal{H}}_{3}&=9\lambda-\delta_0-3\delta_1, \quad
\overline{\mathcal{HF}}=308\lambda-32\delta_0-76\delta_1.
\end{align*}
Now, we will check using the simplest example of fibered surface of genus $3$ that two local signatures $\sigma_{\mathcal{H}_{3}}$ and $\sigma_{\mathcal{HF}}$ associated with $\overline{\mathcal{H}}_{3}$ and $\overline{\mathcal{HF}}$ give different localizations.

\begin{exa}
Let $\{C_{\lambda}\}_{\lambda}\subset |4H_{\mathbb{P}^{2}}|$ be a general Lefschetz pencil of quartics.
The base locus of $\{C_{\lambda}\}_{\lambda}$ consists of $16$ points  and they are on smooth members.
Blowing up at these $16$ points, we obtain a non-hyperelliptic fibration $f\colon S\to \mathbb{P}^{1}$ of genus $3$.
By a simple computation, we get $\chi_f=3$, $e_f=27$, $K_f^2=9$ and $\mathrm{Sign}(S)=-15$.
Note that all singular fibers of $f$ are irreducible curves with one node and the number of them is $27$.
Thus we have $\overline{\mathcal{H}}_3(f)=0$, $\lambda(f)=3$, $\delta_0(f)=27$ and $\delta_1(f)=0$.
Hence we have $\overline{\mathcal{HF}}(f)=60$.
This implies that the number of smooth curves in a general Lefschetz pencil of quartic curves with a hyperflex is $60$.
Let $F_{\mathrm{hf}}$ and $F_0$ respectively be a smooth quartic fiber germ of $f$ with one hyperflex and an irreducible fiber germ of $f$ with one node.
Then clearly we have 
$$
\delta_0(F_{\mathrm{hf}})=0,\quad \delta_1(F_{\mathrm{hf}})=0,\quad \overline{\mathcal{H}}_3(F_{\mathrm{hf}})=0,\quad \overline{\mathcal{HF}}(F_{\mathrm{hf}})=1
$$
and
$$
\delta_0(F_0)=1,\quad \delta_1(F_0)=0,\quad \overline{\mathcal{H}}_3(F_0)=0,\quad \overline{\mathcal{HF}}(F_0)=0.
$$
Thus we get
$$
\lambda_{\mathcal{H}_3}(F_{\mathrm{hf}})=0,\quad 
\lambda_{\mathcal{H}_3}(F_0)=\frac{1}{9},\quad
\sigma_{\mathcal{H}_3}(F_{\mathrm{hf}})=0,\quad 
\sigma_{\mathcal{H}_3}(F_0)=-\frac{5}{9}
$$
and
$$
\lambda_{\mathcal{HF}}(F_{\mathrm{hf}})=\frac{1}{308},\quad 
\lambda_{\mathcal{HF}}(F_0)=\frac{8}{77}, \quad
\sigma_{\mathcal{HF}}(F_{\mathrm{hf}})=\frac{1}{77},\quad \sigma_{\mathcal{HF}}(F_0)=-\frac{45}{77}.
$$
Thus two local signatures $\sigma_{\mathcal{H}_3}$ and $\sigma_{\mathcal{HF}}$ are different.
\end{exa}

Next, let us consider the genus $2$ case.
The rational Picard group of $\overline{\mathcal{M}}_2$ is generated by $\lambda$, $\delta_0$ and $\delta_1$ with one relation $10\lambda=\delta_0+2\delta_1$.
For a semi-stable fiber germ $\widehat{F}$ of genus $2$, we put $\lambda(\widehat{F}):=(\delta_0(\widehat{F})+2\delta_1(\widehat{F}))/10$.
For a not necessarily semi-stable fiber germ $F$, we define $\lambda(F)$ by using the semi-stable reduction similarly as in the previous section.
We also define a (pre-)Horikawa index $\mathrm{Ind}(F):=10\lambda(F)-\delta(F)$ for a relatively minimal genus $2$ fiber germ $F$.
It coincides with the original Horikawa index defined by using the double covering data (cf.\ \cite{Te}, \cite{Ho}, \cite{pi1}) and hence it is non-negative.
A local signature can be defined by $\sigma(F):=4\lambda(F)-\delta(F)$ for any fiber germ $F$ of genus $2$.

Now, we define another local signature for non-bielliptic genus $2$ fiber germs.
Let $\mathcal{B}_2$ be the bielliptic locus on $\mathcal{M}_2$ and $\overline{\mathcal{B}}_2$ its closure in $\overline{\mathcal{M}}_2$.
They are irreducible codimension $1$ loci.
From \cite{FabPa}, the rational linearly equivalence class of $\overline{\mathcal{B}}_2$ is 
$$
\overline{\mathcal{B}}_2=\frac{3}{2}\delta_0+6\delta_1=30\lambda-\frac{3}{2}\delta_0=15\lambda+3\delta_1.
$$
Thus, for non-bielliptic genus $2$ fiber germs, two localizations of the Hodge bundle $\lambda$ can be realized as follows.
We put
$$
\lambda_{\mathcal{B}_2,0}(\widehat{F}):=\frac{1}{30}\overline{\mathcal{B}}_2(\widehat{F})+\frac{1}{20}\delta_0(\widehat{F})
$$
and
$$
\lambda_{\mathcal{B}_2,1}(\widehat{F}):=\frac{1}{15}\overline{\mathcal{B}}_2(\widehat{F})-\frac{1}{5}\delta_1(\widehat{F})
$$
for a semi-stable non-bielliptic fiber germ $\widehat{F}$ of genus $2$.
By using semi-stable reduction, we define $\lambda_{\mathcal{B}_2,0}(F)$, $\lambda_{\mathcal{B}_2,1}(F)$ for any non-bielliptic fiber germ $F$ of genus $2$.
Then $\sigma_{\mathcal{B}_2,i}(F):=4\lambda_{\mathcal{B}_2,i}(F)-\delta(F)$, $i=1,2$ are local signatures for genus $2$ non-bielliptic fibrations.

\begin{exa}
Let $F_0$, $F_1$ and $F_b$ respectively be non-bielliptic genus $2$ fiber germs the image of whose moduli map meets $\Delta_0$, $\Delta_1$ and $\mathcal{B}_2$ transversally (and does not meet other loci among them) at the moduli point of the central fiber.
Then we have
$$
\sigma(F_0)=-\frac{3}{5},\quad \sigma(F_1)=-\frac{1}{5},\quad \sigma(F_b)=0,
$$
$$
\sigma_{\mathcal{B}_2,0}(F_0)=-\frac{4}{5},\quad \sigma_{\mathcal{B}_2,0}(F_1)=-1, \quad \sigma_{\mathcal{B}_2,0}(F_b)=\frac{2}{5},
$$
$$
\sigma_{\mathcal{B}_2,1}(F_0)=-1,\quad \sigma_{\mathcal{B}_2,1}(F_1)=-\frac{9}{5},\quad \sigma_{\mathcal{B}_2,0}(F_b)=\frac{4}{15}.
$$
For example, take a general member $R$ in the complete linear system $|pr_1^{*}\mathcal{O}_{\mathbb{P}^{1}}(N)\otimes pr_2^{*}\mathcal{O}_{\mathbb{P}^{1}}(6)|$, $N\in 2\mathbb{Z}_{>0}$ on $\mathbb{P}^{1}\times \mathbb{P}^{1}$ and construct the double covering $S\to \mathbb{P}^{1}\times \mathbb{P}^{1}$ branched over $R$.
Then the composite $f\colon S\to \mathbb{P}^{1}$ of the double covering and the first projection $pr_1$ is a non-bielliptic fibration of genus $2$.
By a simple computation, we have
$$
\chi_f=N,\quad K_f^{2}=2N,\quad e_f=10N,\quad \mathrm{Sign}(S)=-6N.
$$
Since $R$ is general, we may assume that any singular fiber germ of $f$ is of type $F_0$ as above.
Thus the number of fiber germs of type $F_0$, $F_1$ and $F_b$ is $10N$, $0$ and $15N$, respectively.

\end{exa}

\end{document}